\documentclass[12pt]{article}
\pdfoutput=1
\pdfoutput=1
\usepackage[ansinew]{inputenc}
\usepackage{amssymb,amsthm, amsmath}
\usepackage[ansinew]{inputenc}
\usepackage{amsfonts}
\usepackage{newlfont}
\usepackage{mathrsfs}
\usepackage[all]{xy}
\usepackage{epic}
\usepackage{graphicx}
\usepackage{epstopdf}
\usepackage{color}
 \newif\ifTextures    %Defino unos if's que me serviran para lo de
 \newif\ifnoTextures  %las dos versiones en una

 \newif\ifComentarios %Otro if para los comentarios

 \Texturestrue
 \Texturesfalse

\noTexturestrue %ponemos por defecto que "no estamos en Textures"
 \ifTextures
 \noTexturesfalse
  \fi

  \Comentariostrue
\newcommand{\C}{\mathbb{C}}                       %Los complejos
\newcommand{\Z}{\mathbb{Z}}                         %Los enteros
\newcommand{\N}{\mathbb{N}}                        %Los naturales
\newcommand{\Tt}{\widetilde{\mathcal{T}}}    %Los arboles enraizados
\newcommand{\T}{\mathcal{T}}                         %Los arbres enraizados en el 0
\newcommand{\TT}{\boldsymbol{\T}}     
\newcommand{\gZ}{\mathcal{Z}}                       %El grafo de Cayley de Z
\newcommand{\G}{\mathcal{G}}                        %Un grafo de Cayley de un groupe G
         
\newcommand{\R}{\mathcal{R}}                        %La relaci\\'on
\newcommand{\B}{\mathcal{B}}                        %Un diagrama de Bratteli
\newcommand{\M}{\boldsymbol{M}}                    %espace laminado       
\renewcommand{\L}{\boldsymbol{\mathcal{L}}}                      %La lamination    
\newcommand{\X}{\boldsymbol{X}}              %Conjunto minimal
\newcommand{\Y}{\boldsymbol{Y}}              %Conjunto minimal
\newcommand{\F}{\boldsymbol{\mathcal{F}}}            %La foliaci\'on
\newcommand{\Suc}{\mathcal{S}}                      %El conjunto de s\'{\i}mbolos
\DeclareMathOperator{\len}{long}                      %El operador longitud
\DeclareMathOperator{\iso}{Iso}                         %Isotrop\'{\i}a                   
\newcommand{\red}[2]{ #1 \kern -0.5ex \mid_{#2}}

\newcommand{\norm}[1]{\parallel \! #1 \! \parallel}

\newcommand{\den}[1]{\mid \! #1 \!  \mid}

\renewcommand{\epsilon}{\varepsilon} %ponemos el epsilon bonito
\renewcommand{\leq}{\leqslant}
\renewcommand{\geq}{\geqslant}
\renewcommand{\vec}[1]{\mathbf{#1}}
\newcommand{\inte}[1]{\mathaccent23{#1}}

\newcommand{\aristaA}{\xy (0,0)*{^0};(2,0)*{\bullet};(6,0)*{\bullet}**\dir{-} \endxy}

\newcommand{\aristaC}{\xy (6,0)*{^0};(0,0)*{\bullet};(4,0)*{\bullet}**\dir{-} \endxy}

\newcommand{\aristaB}{\xy (2,0)*{^0};(0,0)*{\bullet};(0,4)*{\bullet}**\dir{-} \endxy}

\newcommand{\aristaD}{\xy (2,4)*{^0};(0,0)*{\bullet};(0,4)*{\bullet}**\dir{-} \endxy}

      \newtheorem{thm}{Th\'eor\`eme}[section]
    \newtheorem{prop}[thm]{Proposition}

    \newtheorem{thrg}[thm]{Th\'eor\`eme de r\'ealisation g\'eom\'etrique}

\theoremstyle{definition}
    \newtheorem{defn}[thm]{D\'efinition}

\def\jour{\small \number\day \space
\ifcase\month\or
 Janvier\or F\'evrier\or Mars\or Avril\or Mai\or Juin\or
 Juillet\or Ao\^ut\or Septembre\or Octobre\or Novembre\ord
D\'ecembre\fi \space\number\year}

\title{Dynamique transverse de la lamination de Ghys-Kenyon\thanks{Financ\'e par 
Ministerio de Ciencia y Tecnolog\'{\i}a BFM2002-04439, Ministerio de Educaci\'on y Ciencia MTM2004-08214 et Universidad del Pa\'{\i}s Vasco UPV 00127.310-E-14790/2002.}}
\author{F. Alcalde Cuesta$^1$, A. Lozano Rojo$^{2}$ et  M. Macho Stadler$^2$
  \bigskip \\ 
 {\footnotesize  $^1$  Departamento de Xeometr\'{\i}a e Topolox\'{\i}a, Universidade de Santiago de Compostela,}  \vspace{-0.7ex} \\ {\footnotesize 15782 Santiago de Compostela (Espagne)} \\
 {\footnotesize $^2$ Departamento de Matem\'aticas, Universidad del Pa\'{\i}s Vasco-Euskal Herriko Unibertsitatea,}  \vspace{-0.7ex} \\  {\footnotesize 48940 Leioa (Espagne)}}
 \date{}
\begin{document}
\maketitle

\section{Introduction}

 Il y a des laminations minimales par surfaces de Riemann o\`u les
types conformes des feuilles se m\'elangent.  Le premier exemple
a \'et\'e construit par  \'E. Ghys \cite{ghys2} \`a partir d'un
arbre ap\'eriodique et r\'ep\'etitif d\'ecrit par R. Kenyon \cite{kenyon}.  La construction comporte deux \'etapes distinctes, valables pour tout sous-graphe r\'ep\'etitif du graphe de Cayley $\G$ d'un groupe infini de type fini $G$. Il s'agit d'abord de construire un espace compact, muni d'un  feuilletage par graphes, puis d'obtenir une lamination par surfaces de Riemann.  Soit $\T = \T(G)$  l'ensemble des sous-graphes infinis de $\G$ contenant l'\'el\'ement neutre $e$ de $G$. On munit $\T$ de la {\em topologie de Gromov-Hausdorff} 
pour laquelle deux sous-graphes de $\G$ sont proches s'ils co\"{\i}ncident sur une grande boule centr\'ee en $e$. Puisqu'une boule ne contient qu'un nombre fini de sous-graphes, un proc\'ed\'e diagonal classique montre que $\T$ est compact. Gr\^ace \`a l'action 
de $G$ sur  $\G$, on d\'efinit une relation d'\'equivalence $\R$ qui identifie un arbre  $T$ et son translat\'e  $T' = g^{-1}.T$ si $g \in T$. On peut d'ailleurs r\'ealiser  $\T$ comme un sous-espace d'un espace m\'etrique compact $\TT= \boldsymbol{\T(G)}$, muni d'un feuilletage par graphes $\F$ dont toutes les feuilles sont rencontr\'ees par $\T$.  Alors $\R$ est induite par  $\F$  et les classes d'\'equivalence sont les ensembles de sommets des feuilles de $\F$.
\medskip

Pour tout graphe $T \in \T$, l'ensemble $X = \overline{\R[T]}$ est un ferm\'e satur\'e pour $\R$, appel\'e l'{\em enveloppe de $T$}. Il est r\'ealisable comme transversale compl\`ete d'un espace feuillet\'e compact $\X$, \`a savoir  la fermeture de la feuille $L_T \in \F$ passant par $T$. Les ensembles $X$ et $\X$ sont minimaux si et seulement si le graphe $T$ est {\em r\'ep\'etitif}, i.e. pour tout nombre r\'eel 
$r > 0$, il existe un nombre r\'eel $R > 0$ tel que toute boule de rayon $R > 0$ contient une boule qui est l'image par translation de la boule de centre $e$ et rayon $r > 0$. Par ailleurs, l'holonomie de $L_T$ est triviale si et seulement si $T$ est {\em ap\'eriodique}, i.e. $T \neq g.T$ pour tout \'el\'ement $g \neq e$ de $G$. Ces d\'efinitions s'inspirent de d\'efinitions analogues pour les pavages \cite{bbg, rw}.
\medskip

$$
\includegraphics[width=2.8in]{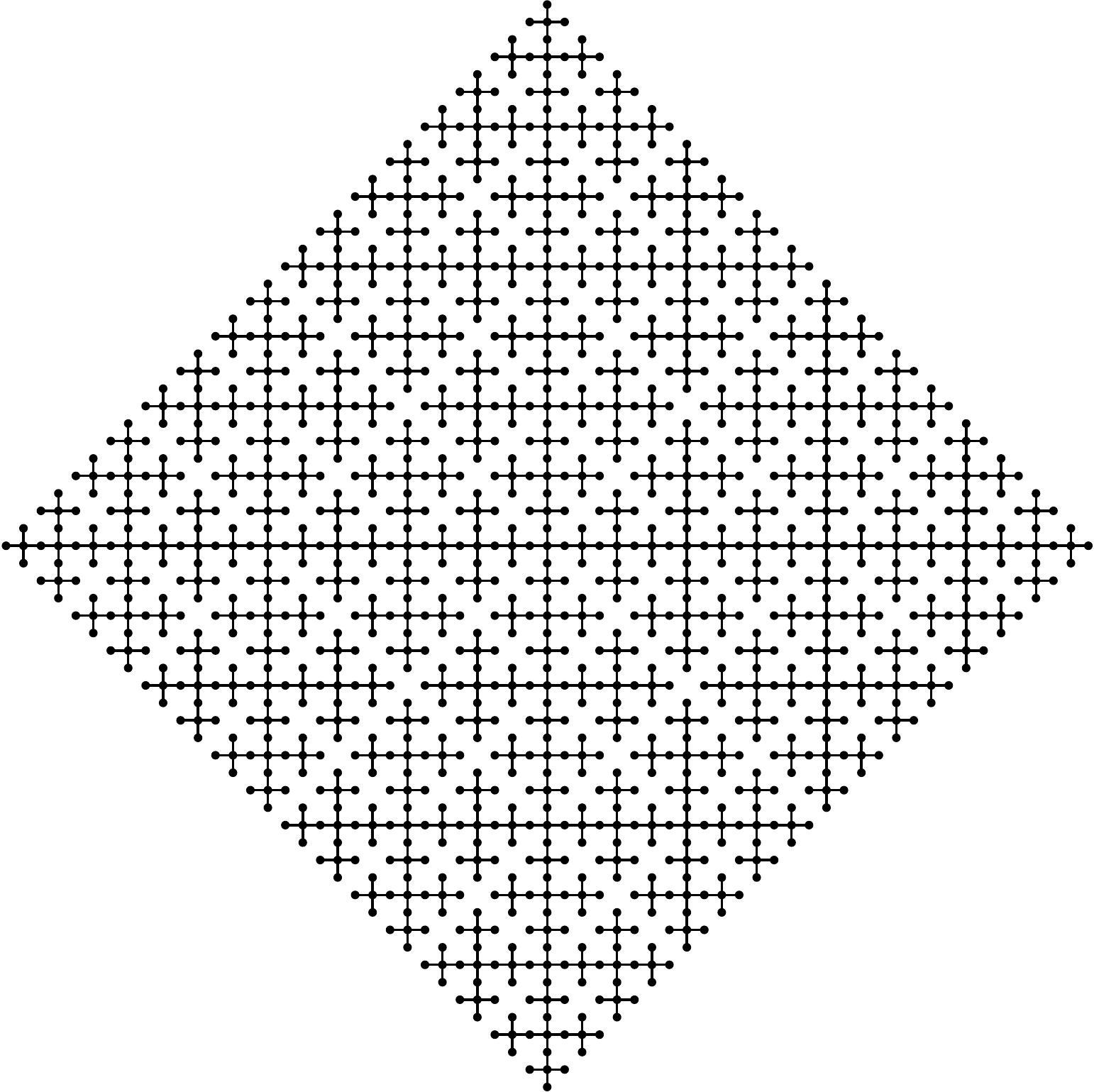}
$$
\begin{center}
Figure 1: Arbre de Kenyon  \label{figKenyon}
\end{center}
\setcounter{figure}{1}
\medskip

Nous appellerons {\em espace feuillet\'e de Ghys-Kenyon} la fermeture $\X$ de la feuille $L_{T_\infty}$ passant par l'arbre de Kenyon $T_\infty$ (voir la figure \ref{figKenyon}) dans l'espace  
$\TT= \boldsymbol{\T({\mathbb Z}^2)}$, 
munie du feuilletage induit par $\F$. Les feuilles sont des sous-arbres r\'ep\'etitifs et ap\'eriodiques du graphe de Cayley $\gZ^2$ de ${\mathbb Z}^2$. En rempla\c{c}ant ces arbres par des surfaces, on obtient la lamination $(\M,\L)$ d\'ecrite par \'E. Ghys. Nous l'appellerons  {\em lamination de Ghys-Kenyon}. Dans ce travail, nous allons r\'ecuperer l'espace $\X$ par un proc\'ed\'e de construction de sous-arbres r\'ep\'etitifs et ap\'eriodiques de $\gZ^2$ \`a partir de suites de $4$ \'el\'ements.  Gr\^ace \`a codage, nous  montrerons que la dynamique transverse de ce feuilletage est repr\'esent\'ee par l'automate suivant: 
\medskip 

\[
\xy 
(-10,0)*{0};
(-10,0)*\xycircle(2.65,2.65){-}="A"; 
"A";
(10,0)*\xycircle(2.65,2.65){-}="B"; 
"B";
(10,0)*{1};
{\ar@/^1pc/(-7.5,1)*{};(7.5,1)*{}};
{\ar@/^1pc/(7.5,-1)*{};(-7.5,-1)*{}};
(-12.8,1)*{}="a";
(-12.8,-1)*{}="b";
"a"; "b" **\crv{(-16,7.5) & (-28,0) & (-16,-7.5)}?(1)*\dir{>}; 
(12.8,1)*{}="a";
(12.8,-1)*{}="b";
"a"; "b" **\crv{(16,7.5) & (28,0) & (16,-7.5)}?(1)*\dir{>}; 
\endxy
\]
\vspace{-1ex}
\begin{center}
Figure 2: Machine \`a sommer binaire \label{fig2adique}
\end{center}
\setcounter{figure}{2}
\bigskip 

\noindent
Cela signifie que $\R$ est {\em stablement orbitalement \'equivalente} \`a la relation engendr\'ee par la somme $S(x) =  x +1$ d\'efinie sur  l'anneau des entiers $2$-adiques, ou de mani\`ere \'equivalente par  la transformation 
$T : \ \{0,1\}^{\N} \to  \ \{0,1\}^{\N}$ 
d\'efinie par:

\noindent i) si $\alpha_0=0$, alors  $T(\alpha)_0=1$ et
$T(\alpha)_n=\alpha_n$ pour tout
 $n\geq 1$, 

\noindent ii) si $\alpha_0=1$, alors $T(\alpha)_0=0$ et
$T(\alpha)_1=T(\sigma(\alpha))_0$ avec $\sigma(\alpha)_n=\alpha_{n+1}$. 
 
\noindent
Nous compl\'eterons l'\'etude de l'exemple en d\'ecrivant sa dynamique topologique.  Nous montrerons ainsi que $\L$ est  {\em affable}, en ce sens que $\R$ est la limite inductive d'une suite de relations d'\'equivalence \'etales compactes  \cite{gps04}.  La dynamique transver\-se de $\L$ sera ainsi repr\'esent\'ee par un syst\`eme dynamique classique. 
\vspace{-1ex}

\section{L'espace feuillet\'e de Gromov-Hausdorff} \label{GH}

Soit $S$ un syst\`eme fini de g\'en\'erateurs de $G$. Le \textit{graphe de Cayley} $\G = \G(G,S)$ est un graphe localement fini non orient\'e, sans boucle, ni ar\^ete multiple, dont les sommets sont les \'el\'ements de $G$. Deux sommets $g_1$ et $g_2$ sont reli\'es par une ar\^ete si $g_1^{-1}g_2 \in S$. On appelle  \textit{longueur} de $g$ le plus petit nombre d'\'el\'ements de $S$ n\'ecessaires pour  \'ecrire $g$, i.e.
$\len_S(g)=\min\{n\geq 1 /  \mbox{$g=s_1...s_n$ avec $s_i s_{i+1} \neq e$} \}$.
La \textit{distance des $S$-mots} est alors donn\'ee par $d_S(g_1,g_2) = \len_S(g_1^{-1} g_2)$
pour tout couple $g_1, g_2 \in G$. Cette distance se prolonge en une distance sur $\G$ telle que toute ar\^ete est isom\'etrique \`a l'intervalle $[0,1]$. Le graphe de Cayley $\G$ devient ainsi un espace m\'etrique connexe par chemins sur lequel le groupe $G$ agit par isom\'etries.

\setcounter{thm}{0}

\subsection{Topologie de Gromov-Hausdorff}

Soit $\T = \T(G)$  l'ensemble des sous-graphes $T$ de $\G$ contenant l'\'el\'ement neutre $e$ de $G$. Notons $B_T(e,N)$ (resp. $\overline{B}_T(e,N)$) la boule ouverte (resp. ferm\'ee) de centre $e$ et de rayon $N$ et  $val_T(e)$ la  valence de $e$, i.e. le nombre d'ar\^etes issues de $e$.
Consid\'erons l'ensemble $A = \{ \ N\geq 1 \ /  \  B_T(e,N)=B_{T'}(e,N) \ \}$ 
 et la quantit\'e
 $$
R(T,T') = \begin{cases}
\ \sup \ A  &  \ \ si \ A \neq \emptyset , \\
\  0  & \ \ si \ A = \emptyset ,
\end{cases}
$$
qui appartient \`a ${\mathbb N} \cup \{+\infty\}$ pour tout couple $T, T' \in \T$.  On d\'efinit alors la \textit{distance de Gromov-Hausdorff} par $d(T,T') = e^{-R(T,T')}$. C'est une ultram\'etrique et donc $\T$ est totalement disconnexe. Puisque  la boule ferm\'ee $\overline{B}_{\G}(e,N)$ ne contient qu'un nombre fini de sous-graphes, un proc\'ed\'e diagonal classique montre que $\T$ est 
compact. Les sous-graphes finis de $\G$ correspondent aux points isol\'es de $\T$. Nous noterons d\'esormais  $\T$ l'ensemble des sous-graphes {\em infinis} de $\G$ contenant l'\'el\'ement neutre $e$. L'avantage de la nouvelle d\'efinition est mise en \'evidence par le fait que $\T = \T(G)$ est alors hom\'eomorphe \`a l'ensemble  de Cantor, sauf si $G = {\mathbb Z}$. 

\setcounter{thm}{0}

\subsection{Structure feuillet\'ee}

L'espace $\T$ est muni d'une relation d'\'equivalence $\R$ qui
identifie deux graphes $T$ et $T'$ si  $T' = g^{-1}.T$ avec $g
\in T$.  Toute classe d'\'equivalence $\R[T]$ peut \^etre alors r\'ealis\'ee comme l'ensemble de sommets d'un graphe $\overline{\R}[T]$. Il suffit de joindre 
$T' = g^{-1}.T$ et $T''= h^{-1}.T$ par une ar\^ete si $d_S(g,h) = 1$. Le graphe $\overline{\R}[T]$ est  donc isomorphe au quotient de $T$ par le groupe de translations $\iso(T)=\{g\in \G / T=g.T\}$. C'est une feuille de l'espace feuillet\'e compact fourni par le r\'esultat suivant:
\medskip

\begin{thrg} \label{th:TRG1}
Il  y a un espace compact, m\'etrisa\-ble et s\'eparable $\boldsymbol\T$, muni d'un feuilletage par graphes $\F$, pour lequel $\T$ est une transversale compl\`ete et $\R$ est la relation d'\'equivalence induite sur $\T$.
\end{thrg}

\begin{proof} Consid\'erons le sous-espace $\Tt = \{ (T,g) \in \T \times \G / \mbox{ $g$ est un sommet de $T$} \}$ de $ \T \times \G$, muni de la pseudodistance
$d((T_1,g_1),(T_2,g_2)) = d(g_1^{-1}.T_1,g_2^{-1}.T_2)$.  Alors $\T$ est le quotient de $\Tt $  par l'action diagonale de $G$ sur 
$\T \times \G$. Chaque classe d'\'equivalence $\R[T]$ est obtenue par passage au quotient \`a partir de l'orbite de $(T,e)$. L'ensemble 
$\widetilde{U}_{(T_1,g_1)} =  \overline{B}_{\Tt }((T_1,g_1), e^{-1})  = \{  (T_2,g_2) \in \Tt  / \overline{B}_{g_1^{-1}.T_1}(e,1) = 
\overline{B}_{g_2^{-1}.T_2}(e,1) \}$
est un ouvert-ferm\'e qui se projette sur l'ouvert-ferm\'e
$U_{g_1^{-1}.T_1} = \overline{B}_{\T}(g_1^{-1}.T_1, e^{-1})$.
Puisque  $\overline{B}_{\G}(e,1)$ ne contient qu'un nombre fini de sous-graphes,  les ensembles $\widetilde{U}_{(T_1,g_1)}$ et $U_{g_1^{-1}.T_1}$ d\'efinissent des partitions finies de $\Tt$ et $\T$ respective\-ment. 
Nous allons  remplacer $\Tt$ par l'ensemble $\boldsymbol{\Tt}$ des couples $(T,x)$ o\`u $x$ est un point quelconque de $T$ qui peut appartenir \`a l'int\'erieur $\inte{\mathrm e}$ d'une ar\^ete $\mathrm e$ de $T$. 
L'application
$\widetilde{\psi}_{(T_1,g_1)} : 
\big( (T_2,g_2),x \big) \in \widetilde{U}_{(T_1,g_1)} \times \overline{B}_{g_1^{-1}.T_1}(e,1)  \mapsto (T_2,g_2.x) \in  \boldsymbol{\Tt}$
est injective en restriction aux ensembles
$\widetilde{U}_{(T_1,g_1)} \times B_{g_1^{-1}.T_1}(e,\frac{1}{2})$ et 
$\widetilde{U}_{(T_1,g_1)}  \times \inte{\mathrm e}$. Leurs images $\widetilde{V}_{(T_1,g_1)}$ et
$\widetilde{V}_{(T_1,g_1)}^{\mathrm e}$ sont munies de topologies telles que les restrictions et leurs inverses 
$\widetilde{\varphi}_{(T_1,g_1)}  : 
\widetilde{V}_{(T_1,g_1)} \to \widetilde{U}_{(T_1,g_1)} \times B_{g_1^{-1}.T_1}(e,\frac{1}{2})$ et
$\widetilde{\varphi}_{(T_1,g_1)} ^{\mathrm e}  :  \widetilde{V}_{(T_1,g_1)}^{\mathrm e} \to \widetilde{U}_{(T_1,g_1)}  \times \inte{\mathrm e}$
sont des hom\'eomorphismes. On munit $\boldsymbol{\Tt}$ de la topologie faible pour laquelle  $\widetilde{V}_{(T_1,g_1)}$ et $\widetilde{V}_{(T_1,g_1)}^{\mathrm e}$ forment un recouvrement ouvert fini. On v\'erifie ais\'ement que:

\noindent
i) l'espace $\Tt$ est r\'ealis\'e comme un sous-espace compact de $\widetilde{\TT}$,  

\noindent
ii) l'action de $G$ sur $\Tt$ s'\'etend en une action de $G$ sur  $\boldsymbol{\Tt}$,  

\noindent
iii) les cartes locales $\widetilde{\varphi}_{(T_1,g_1)}$ et $\widetilde{\varphi}_{(T_1,g_1)}^{\mathrm e}$ forment un atlas feuillet\'e sur $\boldsymbol{\Tt}$ qui d\'efinit un feuilletage par graphes $\widetilde{\F}$ invariant par l'action de $G$.

\noindent
Soit $\TT$  le quotient de
$\boldsymbol{\Tt}$ par l'action de $G$. Alors les applications 
$$
\psi_{g_1^{-1}.T_1} : (g_2^{-1}.T_2,x) \in U_{g_1^{-1}.T_1} \times \overline{B}_{g_1^{-1}.T_1}(e,1)  \mapsto x^{-1}.(g_2^{-1}.T_2) \in  \TT
$$
d\'efinissent des cartes locales 
$$ \varphi_{g_1^{-1}.T_1}  : V_{g_1^{-1}.T_1} \to U_{g_1^{-1}.T_1} \times B_{g_1^{-1}.T_1}(e,\frac{1}{2})
\hspace{1em} \mbox{et} \hspace{1em}
\varphi_{g_1^{-1}.T_1} ^{\mathrm e} :  V_{g_1^{-1}.T_1}^{\mathrm e} \to U_{g_1^{-1}.T_1}  \times \inte{\mathrm e}$$
et donc l'espace $\TT$ poss\`ede  un atlas feuillet\'e fini. Il est  compact car les plaques sont relativement compactes et les transversales sont compactes.  \end{proof}

\setcounter{thm}{0}

\subsection{Structure transverse}

Nous allons pr\'eciser ici la notion de  {\em  dynamique transverse} ({\em mesurable} ou {\em topologique}) utilis\'ee dans la introduction. D'abord, la relation d'\'equivalence $\R$ est d\'efinie par l'action d'un pseudogroupe de transformations $\Gamma$ engendr\'e par les translations $\tau_g : T \mapsto g^{-1}.T$ associ\'ees aux \'el\'ements de $G$. Chacune de ces applications est d\'efinie sur l'ouvert-ferm\'e $D_g = \{ T \in \T / g \in T\} $ de $\T$. Le th\'eor\`eme~\ref{th:TRG1} montre que $\Gamma$ est le {\em pseudogroupe d'holonomie de $\F$ r\'eduit \`a $\T$}. Nous utiliserons donc la notion de {\em dynamique transverse} introduite par A. Haefliger  \cite{haefliger85}. 
\medskip

N\'eanmoins, si l'holonomie est triviale, la dynamique transverse est repr\'esent\'ee par la relation d'\'equivalence induite sur toute transversale compl\`ete. Rappelons qu'une relation d'\'equivalence $\R$ sur un espace bor\'elien standard  $X$ est  {\em mesurable discr\`ete} si les classes d'\'equivalence sont d\'enombrables et si le graphe est un  bor\'elien de $X \times X$. On appelle {\em transformation partielle de $\R$} tout isomorphisme bor\'elien 
$\varphi : A \to B$ entre parties bor\'eliennes de $X$ dont le graphe
$G(\varphi) = \{  (x,y) \in X \times X /  y = \varphi(x)  \} \subset \R$. Une mesure bor\'elienne $\mu$ sur $X$ est dite {\em  invariante pour $\R$} si elle est invariante pour toute transformation partielle $\varphi$, i.e.  $ \mu(\varphi^{-1}(B')) = \mu (B')$
pour tout  bor\'elien $B' \subset B$. La relation d'\'equivalence $\R$ sur $\T = \T(G)$ est mesurable discr\`ete car les classes d'\'equivalence sont d\'enombrables et le graphe de $\R$ est un bor\'elien de $\T \times \T$ en tant que r\'eunion des graphes des transformations partielles $\tau_g$ d\'efinies sur les ouverts-ferm\'es $\overline{B}_\T(T,e^{-1})$ (avec  $g \in \overline{B}_{T}(e,1)$) et de leurs compositions. 

\begin{defn}\label{defdynmes}
Deux relations d'\'equivalence mesur\'ees $(\R,X,\mu)$ et $(\R',X',\mu')$ sont dites:
\medskip

\noindent
i) {\em orbitalement \'equivalentes} si $X$ et $X'$ contiennent des bor\'eliens $Y$ et $Y'$ satur\'es pour $\R$ et $\R'$ et de mesure totale pour lesquels il existe un isomorphisme bor\'elien $\varphi : Y \to Y'$ tel que   $\varphi(\R[x])=\R'[\varphi (x)]$ pour $\mu$-presque tout $x \in Y$ et  
 $f_\ast \mu  \sim \mu'$;
\medskip 

\noindent
ii) {\em stablement orbitalement \'equivalentes} si $X$ et $X'$ contiennent des bor\'eliens $Y$ et $Y'$ dont les satur\'es pour $\R$ et $\R'$ sont de mesure totale tels que les relations d'\'equivalence induites $\R \! \mid_Y$ et $\R' \! \mid_{Y'}$ sont orbitalement \'equivalentes. Nous dirons alors que $\R$ et $\R'$  repr\'esentent une m\^eme {\em dynamique mesurable}.
\end{defn}

Toute relation d'equivalence $\R$ sur un espace bor\'elien ou topologique $X$  est munie d'une structure naturelle de groupo\"{\i}de caract\'eris\'ee par les donn\'ees suivantes:
l'inclusion $\varepsilon : x \in X \mapsto (x,x) \in \R$ de l'espace des unit\'es $X$ dans $\R$, 
les projections $\beta : (x,y) \in \R \mapsto x \in X$ et $\alpha : (x,y) \in \R \mapsto y  \in X$,  l'ensemble des couples composables 
$\R \ast \R = \{  ( (x,y) , (x',y') ) \in \R \times \R  /  \alpha(x,y) = y = x' = \beta(x',y')  \}$,  la multiplication partielle $\mu :   ( (x,y) , (x',y') ) \in \R \ast \R  \mapsto (x,y') \in \R$ et  l'inversion 
$\iota : (x,y) \in \R \to (y,x) \in \R$.
La relation d'equivalence $\R$ est dite  {\em topologique} si elle l'est comme groupo\"{\i}de, c'est-\`a-dire si le graphe de $\R$ est muni d'une topologie (qui en fait un espace localement compact s\'epar\'e) telle que  $\alpha, \beta : \R \to X$ et $\mu: \R \ast \R \to \R$ sont continues et  $\iota : \R \to \R$ est un hom\'eomorphisme. Une telle relation d'\'equivalence est  dite {\em $\beta$-discr\`ete} si  $X$ est ouvert dans $\R$. Pour tout  ouvert $U$ de $\T$ et tout \'el\'ement $g$ de $G$, notons
$O(U,g) = \{ (T,g^{-1}.T) \in \R / T\in U \cap D_g \}$ 
le graphe de la translation $\tau_g$ restreinte \`a $U$. Les ensembles $O(U,g)$ engendrent une topologie sur $\R$, plus fine que celle induite par la topologie produit sur $\T \times \T$, qui en fait une relation d'\'equivalence  topologique $\beta$-discr\`ete.
  
 \begin{defn} \label{defdyntop}
Deux  relations d'\'equivalence $\boldsymbol\beta$-discr\`etes $\R$ et $\R'$ sur $X$ et $X'$ sont dites {\em stablement  orbitalement \'equivalentes} (resp. {\em isomorphes}) si $X$ et $X'$ contiennent des ouverts $Y$ et $Y'$ qui rencontrent toutes les classes d'\'equivalence de $\R$ et $\R'$ tels que les relations d'\'equivalence induites $\R \! \mid_Y$ et $\R' \! \mid_{Y'}$ sont orbitalement \'equivalentes (resp. isomorphes). \end{defn}

\setcounter{thm}{0}

\subsection{R\'ealisation g\'eom\'etrique}

La donn\'ee d'un syst\`eme fini de g\'en\'erateurs $S$ de $G$ fournit un syst\`eme fini de g\'en\'erateurs $\Sigma = \{ \tau_g / g \in S \}$ de $\Gamma$. Par analogie avec l'action d'un groupe, l'orbite  $\Gamma(T) = \R [T]$ est  l'ensemble des sommets d'un graphe $\overline{\Gamma} (T) = \overline{\R}[T]$, muni de la distance $d_\Sigma$ d\'efinie par la longueur des $\Sigma$-mots. Nous dirons alors que $(\R,\T,\Sigma)$ est une {\em relation d'\'equivalence graph\'ee} et que $(\Gamma,\T,\Sigma)$ un {\em pseudogroupe graph\'e}. Dans \cite{alvaro}, le deuxi\`eme auteur  a prouv\'e l'extension suivante du th\'eor\`eme~\ref{th:TRG1}: 

\begin{thrg} \label{th:TRG2}
Soit $\Gamma$ un pseudogroupe de g\'en\'e\-ration compacte agissant sur  un espace localement compact, m\'etrisable et s\'eparable $X$ de dimension $0$. Alors il existe lamination compacte par surfaces de Riemann $(\M,\L)$ dont le pseudogroupe de holonomie est \'equivalent  \`a $\Gamma$.
\end{thrg} 

D'une part, si $Y$ est un ouvert et ferm\'e de $X$ qui rencontre toutes les feuilles et si $\Sigma$ est un syst\`eme de g\'en\'eration compacte pour \mbox{$\Gamma \! \mid_Y$}, alors la fonction de valence 
$val : Y \to \mathbb{N}$ est continue. Il existe donc un espace compact feuillet\'e par graphes $(\Y,\F)$ tel que $Y$ est un ferm\'e qui rencontre toute les feuilles de $\F$ et \mbox{$\Gamma \! \mid_Y$} est le pseudogroupe de holonomie de $\F$ r\'eduit \`a $X$. D'autre part, si $(\Y,\F)$ un espace compact feuillet\'e par graphes  transversalement model\'e par un espace localement compact, m\'etrisable et s\'eparable de dimension $0$, il existe une lamination compacte par surfaces de Riemann $(\M,\L)$ telle que les pseudogroupes de holonomie de $\F$ et $\L$ r\'eduits \`a l'ensemble de sommets $Y$ sont \'egaux. En fait, comme nous l'ont fait remarquer B. Deroin et G. Hector, ce th\'eor\`eme d'\'epaisissement reste valable en dimension topologique finie quelconque. 

\subsection{Graphes r\'ep\'etitifs et ensembles minimaux}

Le but de ce paragraphe est de caract\'eriser les ensembles minimaux de $(\TT,\F)$ et $(\M,\L)$ en adaptant  la {\em propri\'et\'e d'isomorphisme local} des pavages \cite{bbg,rw}.
 
\begin{defn} \label{def:repetitif} 
i) Fixons un couple $T,T' \in \T$. Nous dirons qu {\em $T'$ contient une copie fid\`ele  de la boule $B_T(x,r)$} et nous \'ecrirons 
$B_T(x,r) \hookrightarrow T'$ s'il existe  $g \in G$ tel que 
$g.B_T(x,r) = B_{T'}(g.x,r) \subset T'$.
\medskip 

\noindent
ii) Nous dirons qu'un graphe $T \in \T$ est  \textit{r\'ep\'etitif} si pour tout entier
 $r>0$, il existe un entier
$R >0$ tel que $B_T(x,r) \hookrightarrow B_T(y,R)$
pour tout couple $x,y\in T$.
\end{defn}
\smallskip

\noindent
Nous adaptons ici une version uniforme de la propri\'et\'e d'isomorphisme local usuelle.  En fait, pour les pavages de type fini, les deux propri\'et\'es sont \'equivalentes. L'analo\-gue pour les graphes fait partie du crit\`ere de minimalit\'e suivant (dont l'\'equivalence $(ii) \Leftrightarrow (iii)$ a \'et\'e prouv\'ee dans \cite{B01,ghys2}): 

\begin{thm}
\label{repetitif<=>minimal} Pour tout  $T \in \T$, consid\'erons l'ensemble ferm\'e $X=\overline{\R[T]}$ satur\'e pour $\R$. Les conditions suivantes sont \'equivalentes: 
\medskip

\noindent i) le graphe $T$ est  r\'ep\'etitif;
\medskip

\noindent ii) pour tout 
 $r>0$, il existe  $R>0$ tel que $B_T(e,r)
\hookrightarrow B_T(y,R)$ pour tout
 $y\in T$;
\medskip

\noindent iii) l'ensemble $X$ est  minimal.
\end{thm}

\begin{proof} Il suffit de prouver  $(iii) \Rightarrow (i)$, mais il convient avant de rappeler bri\`eve\-ment
$(iii) \Rightarrow (ii)$. Pour cela,  \`a tout r\'eel $r>0$, on lui associe une suite croissante d'ouverts
$U_R=\{ \ T'\in X \ / \ B_T(e,r)\hookrightarrow B_{T'}(e,R) \ \}$ (avec $R \geq 1$) qui recouvrent $X$. 
Puisque $X$ est compact, il existe $R>0$ tel que $X=U_R$. Pour tout  $x \in T$, le graphe $x^{-1}.T \in U_R$ et donc $B_T(e,r)\hookrightarrow B_{x^{-1}.T}(e,R)$, 
c'est-\`a-dire qu'il existe $g \in G$ tel que:
$g.B_T(e,r)=B_{x^{-1}.T}(g,r)\subset B_{x^{-1}.T}(e,R)$. 
Alors on a:
$$
h.B_T(e,r)= x.B_{x^{-1}.T}(g,r) =  B_{T}(h,r) \subset x.B_{x^{-1}.T}(e,R)=
B_{T}(x,R)
$$
avec $h = xg$ et $B_T(e,r)\hookrightarrow B_T(x,R)$. Pour d\'emontrer $(iii) \Rightarrow (i)$, fixons un r\'eel $r>0$ et un point $x\in T$. Comme auparavant, l'ensemble $X$ est recouvert par une suite croissante d'ouverts  $U_R^x=\{T'\in X / B_T(x,r)\hookrightarrow B_{T'}(e,R)\}$ et  il existe $R>0$, qui d\'epend de $r$ et $x$, tel que $B_T(x,r)\hookrightarrow B_T(y,R)$ pour tout $y\in T$. Pour conclure, il faut pouvoir choisir $R > 0$ ind\'ependant du point  $x$. Remarquons tout d'abord que pour tout sommet $g$ de $\G$, la boule $B_{\G}(g,r) = g.B_{\G}(e,r)$. Rappelons aussi que la compacit\'e $\T$ provient du fait que $B_{\G}(e,r)$ ne contient qu'un nombre fini de sous-graphes. Il en est de m\^eme pour  $B_{\G}(g,r)$. En fait, \`a translation pr\`es, il n'y a qu'un nombre fini de boules de rayon $r >0$ distinctes. Consid\'erons une famille finie de points $x_1,\dots,x_n \in T$ de mani\`ere que les boules $B_T(x_i,r)$ repr\'esentent toutes les classes de translations possibles. Pour tout  
$1\leq i \leq n$ et tout $y\in T$, on a
$B_T(x_i,r) \hookrightarrow B_T(y,R(r,x_i))$.
Si on pose $R=\max\{R(r,x_1),\dots,R(r,x_n)\}$, alors $B_T(x,r)  \hookrightarrow B_T(y,R)$
pour tout couple $x,y \in T$. \end{proof}

\setcounter{thm}{0}

\section{L'espace feuillet\'e de Ghys--Kenyon} \label{GK}

Nous donnons ici une nouvelle construction de l'{\em espace
feuillet\'e de Ghys-Kenyon} \cite{ghys2},  que n'utilise pas  l'{\em arbre de Kenyon} \cite{ghys2,kenyon},  mais ses r\`egles de construction. 

\subsection{L'arbre de Kenyon et l'espace 
feuillet\'e de Ghys--Kenyon}

Nous allons commencer par rappeler la construction de l'arbre de Kenyon. Soient $\gZ^2$ le graphe de Cayley  de $\Z^2$, muni du syst\`eme de g\'en\'erateurs $\{(\pm 1,0),(0,\pm 1)\}$, et  $T_1$ le sous-arbre de $\gZ^2$ d\'ecrit dans la figure 3. 
$$
\hspace*{0.2in} 
\begin{minipage}[t]{0.9in}
\includegraphics[width=0.7in]{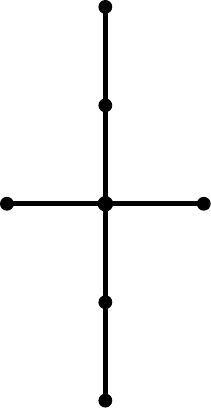}
\end{minipage}
\hspace{0.4in}
\begin{minipage}[t]{1.5in}
\includegraphics[width=1.45in]{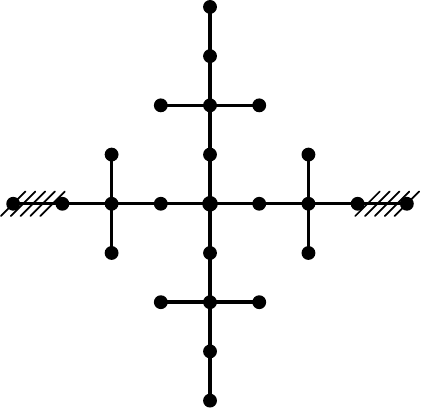}
\end{minipage}
\hspace{0.4in} 
\begin{minipage}[t]{1.5in}
\includegraphics[width=1.45in]{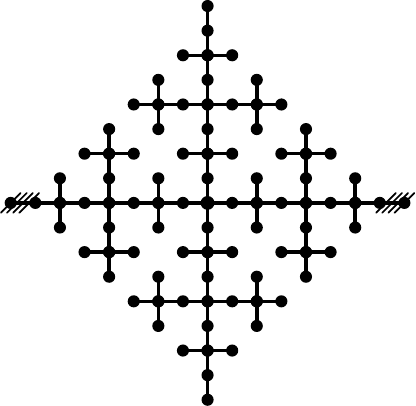} 
\end{minipage}
\setcounter{figure}{3}
$$
\begin{center}
Figure 3: Les arbres $T_1$, $T_2$ et $T_3$  \label{figT1}
\end{center}
\noindent
Cet arbre est translat\'e ensuite par le vecteur $(0,2)$, puis l'image 
est tourn\'ee \`a l'aide des rotations d'angle $\frac{\pi}{2}$, $\pi$ et $\frac{3\pi}{2}$. L'\'elagage des ar\^etes terminales contenues dans l'axe horizontal fournit un arbre $T_2$. Si on r\'ep\`ete ce proc\'ed\'e, on obtient de m\^eme un arbre $T_3$.
Par r\'ecurrence, on obtient une suite d'arbres $T_n$ qui rencontrent les axes horizontal et vertical suivant les intervalles  $[-2^n+1,2^n-1]\times\{0\}$ et  $\{0\}\times [-2^n,2^n]$ 
respectivement. Nous appellerons  \textit{arbre de Kenyon} la r\'eunion 
$T_\infty=\bigcup_{n \geq 1} T_n \subset \gZ^2$. C'est un arbre ap\'eriodique et r\'ep\'etitif ayant $4$ bouts.  
\medskip

%\begin{defn} 
Nous appellerons  \textit{minimal de Ghys-Kenyon}  l'ensemble  $X=\overline{\R[T_\infty]}$.  D'apr\`es le th\'eor\`eme~\ref{th:TRG1},  il existe un  feuilletage par graphes $\F$ d'un espace compact $\X$ pour lequel  $X$  est une transversale compl\`ete et $\R$ est la relation d'\'equivalence induite par $\F$. Nous appellerons \textit{espace feuillet\'e de Ghys-Kenyon} ce minimal de l'espace feuillet\'e de Gromov-Hausdorff $(\TT,\F)$. En fait, d'apr\`es le th\'eor\`eme~\ref{th:TRG2}, on peut remplacer $(\X,\F)$ par une vraie lamination par surfaces de Riemann $(\M,\L)$, appe\-l\'ee \textit{lamination de Ghys-Kenyon}. 
%\end{defn}

\setcounter{thm}{0}

\subsection{Codage des feuilles}

Nous allons reconstruire le minimal de Ghys--Kenyon \`a l'aide d'une application $\Phi:\Suc_4\rightarrow X$ qui,  \`a toute suite $\alpha=\alpha_0\alpha_1 \dots  \in \Suc_4=\{0,1,2,3\}^{\N}=\Z_4^{\N}$, associe un arbre
ap\'eriodique et r\'ep\'etitif $\Phi(\alpha)$ dans l'enveloppe de $T_\infty$. Nous cons\-truirons $\Phi(\alpha)$ de proche en proche en partant du sommet $x_0=0$ et de l'arbre trivial $P_0=\{0\}$. Pour cela, nous commen\c{c}ons par identifier les \'el\'ements de $\Z_4$ avec les racines quatri\`emes de l'unit\'e gr\^ace \`a l'application 
$\vec{r}: \Z_4\rightarrow \C$ d\'efinie par $\vec{r}(k) = e^{\frac{\pi}{2} i k}$. Nous joignons  les sommets $x_0$ et $x_1 = \vec{r}(\alpha_0)$ par une ar\^ete de $\gZ^2$, puis nous prenons la r\'eunion des images de cette ar\^ete par les rotations de centre $x_1$ et d'angle $\frac{\pi}{2}$, $\pi$ et $\frac{3\pi}{2}$. Nous obtenons ainsi un  arbre $P_1 = \Phi(\alpha_0)$. Consid\'erons ensuite l'unique ar\^ete de $\gZ^2$ qui joint  le sommet  $x_2=x_1+2\vec{r}(\alpha_1)$ avec un sommet  de $P_1$.  Nous appelons $P_2 = \Phi(\alpha_0\alpha_1)$ la r\'eunion de l'arbre $P_1$ et  leurs images par les  rotations de centre $x_1$ et d'angle $\frac{\pi}{2}$, $\pi$ et $\frac{3\pi}{2}$.  Par r\'ecurrence, nous avons une suite de sommets $x_n = x_{n-1} + 2^{n-1} \vec{r}(\alpha_{n-1}) = \sum_{i=0}^{n-1}
2^i \vec{r}(\alpha_i)$
et une suite croissante de sous-arbres finis $P_n$ de $\gZ^2$. Alors  $\Phi(\alpha) = \bigcup_{n\geq 0} P_n =  \bigcup_{n\geq 0}
\Phi(\alpha_0 \dots \alpha_{n-1})$ est un arbre
ap\'eriodique et r\'ep\'etitif ayant au plus $2$ bouts.
Nous  appellerons  \textit{squelette de} $\Phi(\alpha)$ la suite
de sommets $x_0x_1\dots x_n\dots$ identifi\'ee au chemin d'ar\^etes obtenu en joignant les  sommets $x_n$ et $x_{n+1}$ par  $2^n$ ar\^etes dans la direction $\vec{r}(\alpha_i)$. Nous venons de d\'efinir une application 
$\Phi : \Suc_4  \to \T$. 
 
\begin{figure}
\begin{center}
\includegraphics[width=3.6in]{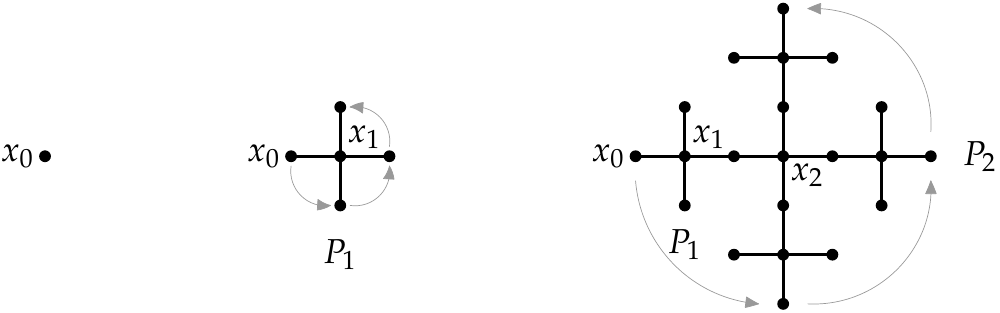}
\vspace{-2ex}
\end{center}
\caption{Construction de l'arbre $\Phi(\alpha)$}  \label{phialpha}
\end{figure}

\begin{prop} \label{thX}
Le minimal de Ghys-Kenyon $X$ est l'enveloppe $\overline{\R[\Phi(\alpha)]}$ de tout arbre cod\'e $\Phi(\alpha)$. Il se d\'ecompose en la r\'eunion disjointe de la classe 
$\R[T_\infty]$ et  de l'ensemble satur\'e $\bigcup_{\alpha\in \Suc_4} \R[\Phi(\alpha)]$.
\end{prop}
\begin{proof} V\'erifions d'abord que $X=\overline{\R[\Phi(\alpha)]}$ pour toute suite
$\alpha\in \Suc_4$. En effet, $T_\infty \in \overline{\R[\Phi(\alpha)]}$  car 
\mbox{$B_{T_\infty}(0,2^n-1)=B_{\Phi(\alpha)}(x_n,2^n-1) - x_n =
B_{\Phi(\alpha)- x_n}(0,2^n-1)$.} Donc $X=\overline{\R[T_\infty]}\subset\overline{\R[\Phi(\alpha)]}$. Mais puisque $\Phi(\alpha)$ est r\'ep\'etitif, on a l'\'egalit\'e. Pour montrer la deuxi\`eme affirmation, on constante  que les arbres $T_\infty$  et $\Phi(\alpha)$ sont  distincts  car ils n'ont pas le m\^eme nombre de bouts. Leurs classes d'\'equivalence $\R[T_\infty]$ et $\R[\Phi(\alpha)]$ le sont  aussi.
Il faut  v\'erifier que tout arbre $T\in X-\R[T_\infty]$ est \'equi\-valent \`a un arbre  $\Phi(\alpha)$. En rempla\c{c}ant  $T$ par un translat\'e $T-v$, nous pourrons supposer que $val(T) = val_T(0) = 1$. Nous construirons alors de proche en proche une suite $\alpha\in \Suc_4$ telle que $T = \Phi(\alpha)$. Psar hypoth\`ese, la sph\`ere 
$S_T(0,1)= \partial \overline{B}_T(0,1)$ 
est  r\'eduite \`a un  point $x_1$ et 
$\alpha_0={\vec{r}}^{-1}(x_1)$. Supposons connus  les codes $\alpha_0\dots\alpha_n$ 
et les points $x_0\dots x_n$ du squelette.  Alors la
sph\`ere $S_T(x_n,2^n) =  \partial
\overline{B}_T(x_n,2^n)$ v\'erifie l'une des deux conditions suivantes:

\noindent
i)  $S_T(x_n,2^n)$ est r\'eduite \`a un seul point 
$x_{n+1}=x_n+2^nv$ o\`u $v\in \Z^4$. Dans ce cas, nous d\'efinirons $\alpha_n=\vec{r}^{-1}(v)$.

\noindent
ii) $S_T(x_n,2^n)$ contient deux points 
$x_{n+1}^0$ et $x_{n+1}^1$. Dans ce cas, il existe $i \in \{ 0,1\}$ tel que $B_T(x_{n+1}^i,2^{n+2}-1)=B_{T_\infty}(0,2^{n+2}-1)$ (voir la figure~5) et nous d\'efinirons
$x_{n+1}=x_{n+1}^{1-i}=x_n+2^nv$ et $\alpha_n={\vec{r}}^{-1}(v)$ avec $v\in \Z^4$.

$$ \vspace*{-3.25cm}
\includegraphics[width=1.3in]{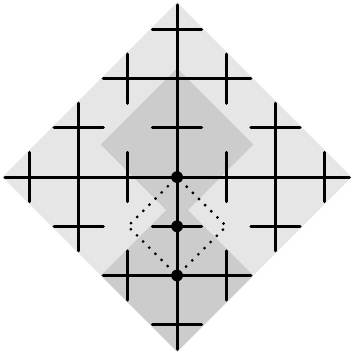}
$$ 
\[
\xy (0,0)*{}; (36,19)*{\scriptstyle x^i_{n+1}};
(22,0)*{\scriptstyle x_n}; (31,-10)*{\scriptstyle x^{1-i}_{n+1}};
{\ar (32,18)*{}; (21,8)*{}}; {\ar (27,-9)*{}; (21,-4)*{}};
 \endxy
\]

\begin{center}
\noindent Figure 5: Le $(n+1)$-i\`eme code \label{figcode}
\end{center}
\setcounter{figure}{5}

\noindent

\noindent
Par r\'ecurrence, nous aurons une suite $\alpha \in \Suc_4$ telle que $T =\Phi(\alpha)$.
\end{proof}

%D'abord, par hypoth\`ese, la sph\`ere 
%\begin{minipage}{7.3cm}
%$S_T(0,1)= \partial \overline{B}_T(0,1)$ est  r\'eduite \`a un  point $x_1$ et 
%$\alpha_0={\vec{r}}^{-1}(x_1)$. Supposons connus  les codes $\alpha_0\dots\alpha_n$ 
%et les points $x_0\dots x_n$ du squelette.  Alors la sph\`ere $S_T(x_n,2^n) =  \partial \overline{B}_T(x_n,2^n)$ v\'erifie l'une des deux conditions suivantes:

%\noindent
%i)  $S_T(x_n,2^n)$ est r\'eduite \`a un seul point $x_{n+1}=x_n+2^nv$ o\`u $v\in \Z^4$. Dans ce cas, nous d\'efinirons $\alpha_n=\vec{r}^{-1}(v)$.

%\noindent
%ii) $S_T(x_n,2^n)$ contient deux points $x_{n+1}^0$ et $x_{n+1}^1$. Dans ce cas, il existe $i = 0,1$ tel que $B_T(x_{n+1}^i,2^{n+2}-1)=B_{T_\infty}(0,2^{n+2}-1)$

%\end{minipage}
%\begin{minipage}{7.5cm}
%$$ \vspace*{-3.5cm}
%\includegraphics{codi-completa}
%$$ 
%\[
%\xy (0,0)*{}; (36,19)*{\scriptstyle x^i_{n+1}};
%(22,0)*{\scriptstyle x_n}; (31,-11)*{\scriptstyle x^{1-i}_{n+1}};
%{\ar (32,18)*{}; (21,8)*{}}; {\ar (27,-10)*{}; (21,-5)*{}};
 %\endxy
%\]

%\begin{center}
%\noindent Figure 5: Le $(n+1)$-i\`eme code \label{figcode}
%\end{center}
%\setcounter{figure}{5}
%\end{minipage}

%\noindent 
% (voir la figure~5) et nous d\'efinirons
%$x_{n+1}=x_{n+1}^{1-i}=x_n+2^nv$ et $\alpha_n={\vec{r}}^{-1}(v)$ avec $v\in \Z^4$.

%\noindent
%Par r\'ecurrence, nous aurons une suite $\alpha \in \Suc_4$ telle que $T =\Phi(\alpha)$.
%\end{proof}

Consid\'erons l'ouvert-ferm\'e 
$X^{\leq 2} = \{  T \in X  / \ val(T) \leq 2 \}$,
 le $G_\delta$ dense
$Y = X - \R[T_\infty] = \bigcup_{\alpha\in \Suc_4} \R[\Phi(\alpha)]$ et 
le bor\'elien $Y^{\leq 2} = Y  \cap X^{\leq 2}$. D'apr\`es la preuve de la proposition ci-dessus, l'application de codage $\Phi : \Suc_4 \to Y^{\leq 2}$ est surjective. 

\setcounter{thm}{0}

\subsection{Relation cofinale}  

Si on munit $\Suc_4 =   \Z_4^\N$ de la topologie produit, engendr\'ee par les cylindres 
$C_{\beta_{0}\dots\beta_{n}}^{i_0\dots i_n}
=  \{\alpha \in \Suc_4 / \alpha_{i_0}=\beta_0,\dots, \alpha_{i_n}=\beta_n \}$, $\Suc_4$ est hom\'eomorphe  \`a l'ensemble de Cantor. Soit 
$\sigma : \Suc_4 \to \Suc_4$ le d\'eplacement de Bernoulli donn\'e par $\sigma(\alpha)_n = \alpha_{n+1}$ pour toute suite
$\alpha \in \Suc_4$ et tout entier $n \geq 0$.  Deux suites $\alpha$ et $\beta$ dans $\Suc_4$ sont {\em cofinales} s'il existe $n \geq 0$ tel que $\sigma^n(\alpha)=\sigma^n(\beta)$, c'est-\`a-dire 
$\alpha_m=\beta_m$ pour tout  $m \geq n$. 
Pour tout couple de suites finies $\alpha_0\dots\alpha_n$ et $\beta_0\dots\beta_n$, les arbres finis $\Phi(\alpha_0\dots\alpha_n)$ et  $\Phi(\beta_0\dots\beta_n)$ sont reli\'es par $\Phi(\beta_0\dots\beta_n) = \Phi(\alpha_0\dots\alpha_n) - v$ 
o\`u le vecteur
$ v = \sum_{i=0}^{n} 2^i \vec{r}(\alpha_i) -
             \sum_{i=0}^{n} 2^i \vec{r}(\beta_i)
          = \sum_{i=0}^{n} 2^i \big( \vec{r}(\alpha_i) -
             \vec{r}(\beta_i)\big)$.
Un argument simple montre alors que:

\begin{prop}  \label{Phicompatible} Deux arbres cod\'es $\Phi(\alpha)$ et $\Phi(\beta)$ sont 
$\R$-\'equivalents si et seulement si les suites $\alpha$ et $\beta$
sont cofinales.
\end{prop}

La remarque pr\'ec\'edente montre aussi que $\Phi$ est injective, ce qui nous donne:  

\begin{prop} L'application 
$\Phi : \Suc_4 \to Y^{\leq 2}$ est une bijection
\end{prop}

Signalons que l'expansion binaire des \'el\'ements de $\Z_4$ fournit un hom\'eomor\-phisme
entre  $\Suc_4 = \Z_4^\N $ et  $\Suc_2 = \Z_2^{\N}$, induit par les substitutions 
$0\rightarrow 00$, $1\rightarrow 10$,  $2\rightarrow 01$ et  $3\rightarrow 11$
obtenues en rempla\c{c}ant  $k \in  \Z_4$ par un couple d'\'el\'ements $a(k)$ et $b(k)$ de $\Z_2$ tels que  $k = a(k) + 2b(k)$. \'Evidemment les relations cofinales sur $\Suc_4 = \{0,1,2,3\}^\N $ et $\Suc_2 = \{0,1\}^{\N}$ deviennent  isomorphes.  Il y a d'ailleurs une \'equivalence orbitale entre la relation cofinale $\R_{\mathrm{cof}}$ sur  $\Suc_2$ et la relation d'\'equivalence engendr\'ee par la transformation $T : \ \{0,1\}^{\N} \to  \ \{0,1\}^{\N}$ d\'ecrite dans l'introduction.
Sauf les suites $000\dots$ et $111\dots$ qui appartiennent \`a une m\^eme orbite, les classes de cofinalit\'e co\"{\i}ncident avec les orbites de $T$. 

\setcounter{thm}{0}

\subsection{Dynamique bor\'elienne} \label{sec:Phi}

Emprunt\'ee de  la th\'eorie des pavages, la notion de {\em motif} est le bon outil pour d\'ecrire la $\sigma$-alg\`ebre des bor\'eliens de  $X$.  Tout sous-arbre fini $P$ de $\gZ^2$ contenant l'origine sera appel\'e un \textit{motif} de $\gZ^2$. Nous dirons que $T \in \T$ \textit{contient le motif $P$ autour d'un sommet $p$} si $P+p \subset T$ et nous d\'efinirons  
$X_{P,p}  = \{T\in X / P\ + p \subset  T\}$. Si $p=0$, nous \'ecrirons simplement $X_P$. Comme pour les pavages \cite{bbg}, les ensembles $X_P$ sont des ouverts-ferm\'es  de $X$.
 N\'eanmoins, les motifs ne suffisent pas pour engendrer la topologie de $X$. En effet, la boule $B = \overline{B}_X(\Phi(00\dots),e^{-1})$ est l'ensemble des arbres $T$ tels que $\overline{B}_T(0,1) = \overline{B}_{\Phi(00\dots)}(0,1) = \aristaA$, mais il n'y a aucun motif  $P$ tel que $X_P \subset B$. En fait, 
$$B = X_{\text\aristaA} - \big(
X_{\text\aristaB}\cup X_{\text\aristaC}\cup
X_{\text\aristaD}\big). \vspace{-2ex}$$
En g\'en\'eral, pour tout arbre $T \in X$ et tout entier $r > 0$, la  boule  $\overline{B}_X(T,e^{-r})$ est l'ouvert-ferm\'e $X_{(P,A)}$ associ\'e au {\em motif fin} $(P,A)$ o\`u $ P = \overline{B}_T(e,r)$ et  $A$  est l'ensemble des ar\^etes du graphe  
$\overline{B}_{\gZ^2}(e,r) - B_T(e,r-1)$ qui rencontrent $P$. Par cons\'equent, les ouverts-ferm\'es $X_P$ engendrent la $\sigma$-alg\`ebre des  bor\'eliens. 
\medskip 

L'application de codage $\Phi$ n'est pas continue, car $S_4$ est compact, mais $Y^{\leq 2}$ ne l'est pas. N\'eanmoins,  $\Phi$ a deux  propri\'et\'es importantes: 
 
\begin{prop}\label{Phimeibleabierta} L'application $\Phi: \Suc_4\rightarrow
Y^{\leq 2}$ est bor\'elienne ouverte.
\end{prop}

\begin{proof} Pour tout motif $P$, l'ensemble $\Phi^{-1}(X_P) = \bigcup_{\alpha_0\dots\alpha_n\in\mathcal{P}}
C_{\alpha_{0}\dots\alpha_{n}}^{0 \dots n}$ o\`u $\mathcal P = \{ \alpha_0\dots\alpha_n / P
\subset  \Phi(\alpha_0\dots\alpha_n) \}$. Par ailleurs, on a
$\Phi(C_{\alpha_{0}\dots\alpha_{n}}^{0 \dots n})=X_{\Phi(\alpha_0\dots\alpha_n)}$. \end{proof}

Nous pouvons maintenant affirmer  que {\em la dynamique transverse bor\'elienne de la lamination de Ghys-Kenyon est repr\'esent\'ee par une machine \`a sommer binaire}. 

\setcounter{thm}{0}

\subsection{Propi\'et\'es  ergodiques} \label{sec:medida}

Soit  $\R$ une relation d'\'equivalence mesurable discr\`ete sur $X$, munie  d'une mesure quasi-invariante ergodique $\mu$. Par analogie avec  la classification des facteurs de F. J. Murray et J. von Neumann, on peut distinguer trois types de relations: 
\smallskip 

\noindent
1) {\em Type $\mathrm{I}_n$ (avec $n = 1,2, \dots, \infty$)}: si $\R$ est transitive (avec cardinal $\# X = n$).
\smallskip 

\noindent
2) {\em Type $\mathrm{II}_n$ (avec $n = 1 \mbox{ ou }  \infty$)}: si  $\R$ n'est pas transitive et  si $\mu$ est \'equivalente \`a une mesure (finie ou infinie) invariante pour $\R$.
\smallskip 

\noindent
3) {\em Type} $\mathrm{III}$: s'il n'existe pas de  mesure invariante \'equivalente \`a $\mu$.
 \smallskip 
 
\noindent
Si $\mu_4$ est la mesure de probabilit\'e \'equidistribu\'ee sur $\Suc_4$, alors $\R_{\mathrm{cof}}$ est de type $\mathrm{II}_1$. 

\begin{prop} \label{thtipoII1}
La relation d'\'equivalence $\R$ sur $X$ est de type
$\mathrm{II}_1$.
\end{prop}

\begin{proof} Pour tout  $n\geq 1$, notons $B_n$ la boule de centre 
$T_\infty$ et de rayon $n$ contenue 
dans  $\R[T_\infty]$. L'isomorphisme  entre  $\overline{\R}[T_\infty]$ et $T_\infty$ identifie $B_n$ avec  $B_{T_\infty}(0,n)$.  Soit $\mu_n$ la mesure de comptage sur $B_n$.  Pour tout  motif $P$, on a:
$$
\mu_n(X_P) = \frac{\#  B_n \cap X_P}{\#  B_n} =
  \frac{\#  \{ \ p \in B_{T_\infty}(0,n) \ / \ P + p \subset T_\infty \ \}}{\#  B_{T_\infty}(0,n))}
=   \frac{A(P,n)}{V(n)}
$$
Quitte \`a extraire une sous-suite, on peut supposer que $\mu_n$ converge faiblement vers une mesure  de probabilit\'e  $\mu$. Puisque $X_P$ est un ouvert-ferm\'e, on a:
$$
\label{defmu} \mu(X_P) = \lim_{n \to \infty} \mu_n(X_P) =
 \lim_{n\rightarrow\infty}
  \frac{A(P,n)}{V(n)} =  \mbox{\emph{fr\'equence du motif $P$}}.
$$
D'autre part, pour tout
sommet $v \in P$, l'ensemble 
$X_P - v  = X_{P-v}$ est l'image de $X_P$ par la 
 translation $\tau_v(T) =  T-v$. 
Si  $T_\infty$ contient le motif $P$ autour d'un point $p \in B_{T_\infty}(0,n-r)$, il contient aussi
le motif  $P-v$ autour du point  $p + v \in  B_{T_\infty}(0,n)$ avec $r > \norm{v}$.
Donc \vspace{-1ex}
$$
\den{\mu_n(X_P-v)-\mu_n(X_P)}   \leq   \frac{V(n) - V(n-r)}{V(n)}  \leq  \frac{V(n+r) - V(n-r)}{V(n)}
$$
pour tout $n \in \N$. Mais puisque la fonction  $V(n)$ est \`a croissance sous-exponentielle, il vient $ \lim_{n \to \infty} \den{\mu_n(X_P-v)-\mu_n(X_P)}  =  0$ et donc $\mu$ est invariante pour $\R$. 
   \end{proof}

\begin{prop} \label{thequivorbest}
L'application $\Phi$ d\'efinit une \'equivalence
orbitale stable entre les relations d'\'equivalence mesur\'ees $\R_{\mathrm{cof}}$ sur $\Suc_4$ et $\R$ sur  $X$.
\end{prop}

\begin{proof} Puisque le satur\'e 
de $Y^{\leq 2}$ est de mesure totale, il nous suffit de d\'emontrer que $\Phi : \Suc_4 \to Y^{\leq 2}$ envoie $\mu_4$ sur une mesure \'equivalente \`a $\mu  |_{Y^{\leq 2}}$. Par l'invariance de $\mu$, on a
$\mu(X^{\leq 2}) =\frac{3}{4}$ et donc $\mu_{X^{\leq 2}} = \frac{4}{3} \mu |_{X^{\leq 2}}$ est une mesure de probabilit\'e sur $X^{\leq 2}$ invariante pour $\R \!
\mid_{X^{\leq 2}}$. L'inverse de $\Phi$  envoie la mesure induite par $\mu_{X^{\leq 2}}$ sur une mesure de probabilit\'e sur $\Suc_4$ invariante pour $\R_{\mathrm{cof}}$. L'unicit\'e ergodique de $\mu_4$ entra\^{\i}ne que
$\Phi_\ast \mu_4 = \mu_{X^{\leq 2}}  |_{Y^{\leq 2}}$.   \end{proof}

\begin{thm} \label{dinmed}
La dynamique transverse mesurable  de la lamination de Ghys-Kenyon $(\M,\L)$ est r\'epresent\'ee par une machine \`a sommer  binaire. En outre, elle est uniquement ergodique. 
\end{thm}

\noindent
Un tr\`es joli r\'esultat d'\'E. Ghys  \cite{ghys1} permet de parler du type topologique des feuilles g\'en\'eriques de $\L$. De notre cas, on a que:
\smallskip

{\em
\noindent 
i)  il y a un ensemble satur\'e r\'esiduel et de mesure totale dont toutes les feuilles ont exactement un bout; 
\smallskip

\noindent ii)  il y a un ensemble satur\'e maigre et de mesure nulle constitu\'e par une infinit\'e non d\'enombrable des feuilles ayant deux bouts; 
\smallskip

\noindent iii) il y a une seule feuille avec quatre bouts.}
\smallskip 

\noindent
Le point essentiel est de v\'erifier qu'il y a  correspondance biunivoque entre l'ensemble des feuilles ayant deux bouts et l'ensemble des suites de $\Suc_4$ contenant un nombre fini de d\'etours et une infinit\'e d'aller et retours. Pour toute suite $\alpha\in\Suc_4$, nous appelons {\em aller et retour} (resp. {\em d\'etour})  tout couple $\alpha_n\alpha_{n+1}$ avec $\alpha_n \neq \alpha_{n+1}$ ayant la m\^eme (resp. distincte) parit\'e. Cela permet de montrer que l'ensemble des feuilles \`a deux bouts est non d\'enombrable de mesure nulle. D'apr\`es le lemme 2.6 de \cite{B01}, l'ensemble des feuilles  ayant  un bout est r\'esiduel. 

\setcounter{thm}{0}

\section{Dynamique topologique }  \label{DT}

Toutes les $\R$-classes du minimal de Ghys-Kenyon sont obtenues \`a partir des m\^emes motifs par un m\^eme proc\'ed\'e d'inflation. Nous utiliserons l'inclusion de ces motifs dans les motifs qui r\'esultent de l'inflation pour d\'ecrire sa dynamique topologique.

\subsection{Relations d'\'equivalence affables}

Une relation d'\'equivalence $\beta$-discr\`ete $\R$ sur un espace localement compact  s\'epar\'e $X$ est dite {\em  compacte} \cite{gps04} si $\R - \Delta$ est compact o\`u $\Delta$ est la diagonale de 
$X \times X$. 

\begin{defn}[\cite{gps04}] Une relation d'\'equivalence $\R$ d\'efinie sur un espace totale\-ment disconnexe $X$ est dite {\em affable} s'il existe une suite croissante de relations d'\'equivalence compactes $\R_n$ telle que $\R=\bigcup_{n \in {\mathbb N}} \R_n$. Si on munit $\R$ de la topologie limite inductive, alors 
$\R = \varinjlim \R_n$ est une relation d'\'equivalence $\beta$-discr\`ete
{\em approximativement finie} (AF en abr\'eg\'e).
\end{defn}

Un {\em diagramme de Bratteli} est un graphe orient\'e $\B = (V,E)$ dont les ensembles de sommets et d'ar\^etes admettent des d\'ecompositions $V = \bigsqcup_{n\geq 0} V_n$ et $E = \bigsqcup_{n\geq 0} E_n$ o\`u $V_n$ et $E_n$ sont des ensembles finis non vides tels que  pour toute ar\^ete $e \in E_n$,  l'origine $\alpha(e)  \in V_n$ et  l'extremit\'e $\beta(e)\in V_{n+1}$  \cite{gps04}. On appelle {\em source} tout sommet  $v$ tel que $\beta^{-1}(v) = \emptyset$. Soit  $X_\B$ l'espace des chemins infinis $e_ne_{n+1}e_{n+2}\dots$ (avec  $\alpha(e_{i+1}) = \beta(e_i)$) issus d'une source $ \alpha(e_n)$ de $\B$.  La relation d'\'equivalence  \emph{cofinal} $\R_\B$ sur $X_\B$ (qui identifie $e_ne_{n+1}\dots$ et $e'_m,e'_{m+1}\dots$ s'il existe $N \geq m,n$ tel que $e'_i=e_i$ pour tout $i\geq N$) est affable. En fait, d'apr\`es \cite{gps04}, toute relation d'\'equivalence  AF sur $X$ est isomorphe \`a la relation cofinale $\R_\B$ sur $X_\B$ d\'efinie par un diagramme de Bratteli $\B$.

\setcounter{thm}{0}

\subsection{Affabilit\'e du minimal de Ghys-Kenyon} \label{sec:af}

Soit  $\mathcal{P}_n = \{ A_n^0, A_n^1 , A_n^2, A_n^3, B_n^0,B_n^1, C_n \}$
la famille de {\em motifs basiques de taille  $n$} d\'efinis par $A_n^k =   \overline{B}_{T_\infty}(0,2^n-1) \cup e^k_n$, 
$B_n^k =  \overline{B}_{T_\infty}(0,2^n-1) \cup e^k_n \cup e^{k+2}_n$ et 
$C_n =  \overline{B}_{T_\infty}(0,2^n)$
o\`u $e^k_n$ est l'ar\^ete qui r\'elie $(2^n-1)\vec{r}(k)$ et $2^n\vec{r}(k)$ pour tout $k \in \Z_4$ (voir la figure~6).  Deux \'el\'ements $T$ et $T'$ de $X - X_{C_n}$ sont 
{\em $\R_n$-\'equivalents} s'il existe un motif  basique $P\in\mathcal{P}_n-\{C_n\}$ et deux sommets $v, v'\in P$ avec $\norm{v},\norm{v} < 2^n$ tels  que $P \subset T-v = T'-v'$. D'autre part,  la relation $\mathcal{R}_n$ est triviale sur $X_{C_n}$. 

$$
\hspace{1.9cm}
\includegraphics[width=1.2in]{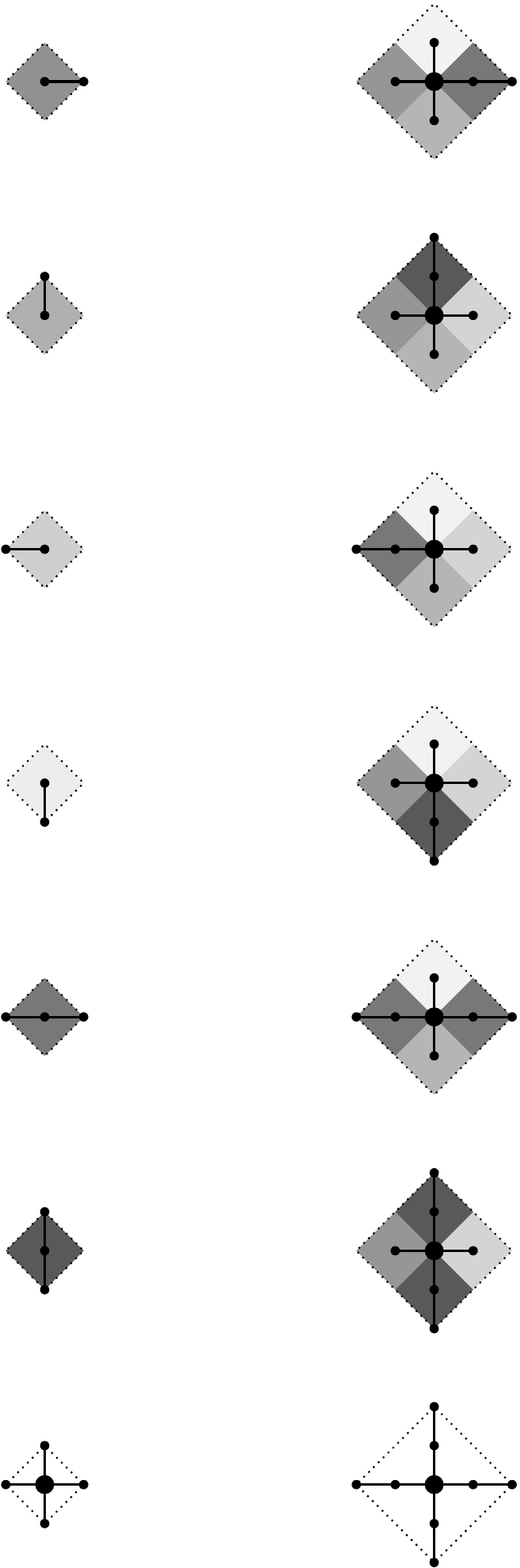}
\hspace{-4.7cm}
\begin{array}[b]{ll}
A_0^0 \hspace{4.9cm} & A_1^0 =  C_0 + A_0^0 + A_0^1 + A_0^3 + B_0^0 \vspace{0.87cm} \\
A_0^1  \hspace{4.9cm} & A_1^1 = C_0 + A_0^0 + A_0^1 + A_0^2 + B_0^1\vspace{0.87cm}  \\
A_0^2  \hspace{4.9cm}  & A_1^2 = C_0 + A_0^1 + A_0^2 + A_0^3 + B_0^0 \vspace{0.87cm} \\
A_0^3  \hspace{4.9cm} & A_1^3 = C_0 + A_0^0 + A_0^2 + A_0^3 + B_0^1 \vspace{0.87cm} \\
B_0^0   \hspace{4.9cm} & B_1^0 = C_0 + A_0^1 + A_0^3 + 2\,B_0^0 \vspace{0.87cm}  \\
B_0^1  \hspace{4.9cm} & B_1^1 = C_0 + A_0^0 + A_0^2 + 2\,B_0^1 \vspace{0.87cm} \\
C_0  \hspace{4.9cm} & C_1 = C_0 + 2B_0^0 + 2 B_0^1 \vspace{1ex}
\end{array}
$$

\begin{center}
Figure 6: Les familles $\mathcal{P}_0$ et $\mathcal{P}_1$ et les r\`egles
d'inflation
\end{center}
\setcounter{figure}{6}

\begin{prop}
Les relations d'\'equivalence $\R_n$ sont compactes et ouvertes dans $\R$ et donc $\R_\infty = \bigcup_{n \in {\mathbb N}} \R_n$ est affable et  ouverte dans  $\R$. 
\end{prop}

\begin{proof} Montrons que $\R_n$ est ouverte dans $\R$. Pour tout couple  $(T,T' ) \in \R_n$, il existe un motif  $P \in \mathcal{P}_n$  et deux sommets $v,v' \in P$ avec $\norm{v}, \norm{v'} < 2^n$ tels que $P \subset T- v=T'-v'$. Choisissons $N > 0$ tel que  $P + v  \subseteq B_T(0,N)$, puis consid\'erons l'ouvert $U = \{  T''  \in X  /  B_{T''}(0,N)=B_T(0,N)  \}$ de $X$ et l'ouvert 
$O(U,w)$ de $\R$ o\`u $w = v-v'$. Pour tout $T''  \in U$, le couple $(T'',T''-w) \in \R_n$ car $T''$ contient  le motif $P$ autour de $v$. Donc $(T,T') \in O(U,v) \subset \R_n$.  Alors $\R_n$ est la r\'eunion des ouverts $O(U,w)$ associ\'es motifs $P \in \mathcal{P}_n$ et aux sommets $v,v' \in P$ tels que  $\norm{v}, \norm{v'} < 2^n$. En rempla\c{c}ant $U$ par l'ouvert-ferm\'e $X_{P,v}$ et $O(U,w)$ par le graphe de la translation $T'' \mapsto T''-w$ d\'efinie sur $X_{P,v}$, nous aurons que $R_n$ est compacte. \end{proof}

Toutes les classes d'\'equivalence de $\R$ et $\R_\infty$ sont \'egales, sauf celle de $T_\infty$ qui se d\'ecompose en la r\'eunion de la classe triviale $\{T_\infty\}$ et de quatre classes isomorphes aux composantes connexes de $T_\infty-\{0\}$. La dynamique topologique de $\R_\infty$ est  repr\'esent\'ee par le diagramme de Bratteli $\B = (V,E)$ o\`u 
$V_0=\{ 0\}$,  $V_{n+1} = 
\mathcal{P}_n = \{ A_n^0, A_n^1 , A_n^2, A_n^3, B_n^0,B_n^1, C_n \}$ et $P \in \mathcal{P}_n$ est reli\'e par une ar\^ete de $E_{n+1}$ \`a $Q  \in \mathcal{P}_{n+1}$ si et seulement si  $Q$ contient une copie fid\`ele de $P$. 
L'isomorphisme 
$\Psi : X \to X_\B$ entre $\R_\infty$ et  $\R_\B$ est donn\'e par $\Psi (T) = (e_0,e_1,\dots )$ o\`u $\beta(e_n)$ est l'unique motif  $P \in \mathcal{P}_{n+1}$ pour lequel $T-v$ appartient \`a l'ouvert-ferm\'e $X_{(P,A)}$ avec $v \in P$ et $A$ form\'e des ar\^etes de $\overline{B}_{T_\infty}(0,2^{n+1})$ qui n'appartiennent pas \`a $P$.
Pour tout $T \in X$ avec $val(T) =4$, l'origine $0$ est l'intersection des translat\'es de quatre motifs basiques de taille $n$. Nous modifierons  alors $\R_n$ pour que $0$ devienne \'equivalent aux autres points du translat\'e de $A_n^0$ ou de $B_n^0$. Nous obtiendrons ainsi une suite de relations d'\'equivalence compactes 
$\R'_n \supset \R_n$. Alors $\R'_\infty = \bigcup_{n \in {\mathbb N}} \R'_n$ est affable. Puisque les bouts de la  feuille de $\F$ passant par $T_\infty$ sont partout denses,  $\R[T_\infty]$ se d\'ecompose en la r\'eunion de quatre orbites denses et donc $\R'_\infty$ est minimale. Nous pouvons maintenant appliquer le corollaire 4.17 de \cite{gps04}: 
 
\begin{thm} La relation d'\'equivalence $\R$ est affable et la dynamique trans\-verse de la lamination de Ghys-Kenyon est  r\'epresent\'ee par un syst\`eme dynamique minimal sur l'ensemble de Cantor. 
\end{thm}

\end{document}